\documentclass[12pt,reqno]{amsart}
\usepackage{amsmath,amsthm,amssymb,amsfonts,amscd}
\usepackage{mathrsfs}%mathtools
\usepackage{bbding}
\usepackage[backref=page]{hyperref}
\usepackage{geometry}
\usepackage{color}
\usepackage{xcolor}

\usepackage{picture,epic}
\usepackage{tikz}

\usepackage{etoolbox}
\usepackage{orcidlink}

\newif\ifrevisionmode
\revisionmodetrue
\numberwithin{equation}{section}

\theoremstyle{plain}
\newtheorem{theorem}{Theorem}[section]
\newtheorem{lemma}[theorem]{Lemma}
\newtheorem{corollary}[theorem]{Corollary}
\newtheorem{proposition}[theorem]{Proposition}

\theoremstyle{definition}
\theoremstyle{remark}
\newtheorem{remark}[theorem]{Remark}

\newcommand{\PGL}{\operatorname{PGL}}
\newcommand{\Tr}{\operatorname{Tr}}
\newcommand{\Nrd}{\operatorname{Nrd}}
\newcommand{\F}{\mathbb F}
\newcommand{\Acal}{\mathcal A_p}
\newcommand{\Bcal}{\mathcal B}

\makeatletter
\def\@tocline#1#2#3#4#5#6#7{\relax
  \ifnum #1>\c@tocdepth % then omit
  \else
    \par \addpenalty\@secpenalty\addvspace{#2}%
    \begingroup \hyphenpenalty\@M
    \@ifempty{#4}{%
      \@tempdima\csname r@tocindent\number#1\endcsname\relax
    }{%
      \@tempdima#4\relax
    }%
    \parindent\z@ \leftskip#3\relax \advance\leftskip\@tempdima\relax
    \rightskip\@pnumwidth plus4em \parfillskip-\@pnumwidth
    #5\leavevmode\hskip-\@tempdima
      \ifcase #1
       \or\or \hskip 1em \or \hskip 2em \else \hskip 3em \fi%
      #6\nobreak\relax
    \hfill\hbox to\@pnumwidth{\@tocpagenum{#7}}\par% <---- \dotfill -> \hfill
    \nobreak
    \endgroup
  \fi}
\makeatother

\begin{document}
\title{PageRank on Lubotzky--Phillips--Sarnak graphs}

\author{Shenghao Hua~\orcidlink{0000-0002-7210-2650}}
\address[1]{Shanghai Institute for Mathematics and Interdisciplinary Sciences (SIMIS), Shanghai, 200433, China}
\address[2]{Research Institute of Intelligent Complex Systems, Fudan University, Shanghai, 200433, China}
\email{huashenghao@vip.qq.com}

\begin{abstract}
We compare return probabilities for random walks stopped at an independent geometric time on Lubotzky--Phillips--Sarnak graphs
and on the infinite regular tree, with the McKay measure giving
the tree value and the Ihara zeta function describing the additional
returns caused by cycles.
We estimate this correction when the mean walk length is comparable to the length of the shortest cycle, a regime in which walks long enough to traverse a cycle receive non-negligible weight.

Unconditionally, we prove an upper bound with a double logarithmic factor and bounded fixed moments over prime levels after multiplication by the number of vertices.
Under the Generalized Riemann Hypothesis for quadratic Dirichlet $L$-functions, the correction is of the same order as the reciprocal
of the number of vertices.
\end{abstract}

\date{}

\maketitle
%\setcounter{tocdepth}{1}%使目录只显示到节
%\tableofcontents

%%%%%%%%%%%%%%%%%%%%%%%%%%%%%%%%%%%%%%%%%%%%%%%%%%%%%%%%%%%%%%%%
%%%%%                        Section                       %%%%%
%%%%%%%%%%%%%%%%%%%%%%%%%%%%%%%%%%%%%%%%%%%%%%%%%%%%%%%%%%%%%%%%

\section{Introduction}

The infinite regular tree gives a natural reference for the spectrum of a finite regular graph.
As the girth tends to infinity with the degree fixed, the eigenvalue distribution converges to the Kesten--McKay law \cite{McKay1981}.
A quantitative comparison asks how much the cycles of a finite graph change its spectral averages.

Closed walks give a direct approach to this question.
Returns on the tree project to returns on the graph, while cycles account for the additional returns.
McKay expressed this contribution through non-backtracking walks in his work on spanning trees \cite{McKay1983}.
Serre later related limiting spectral distributions to closed-walk counts and Ihara zeta functions \cite[Section~8]{Serre1997}.
The comparison with the McKay measure also underlies the variance conjecture of Nestoridi and Sarnak for non-backtracking walks \cite{NestoridiSarnak2023}.
Sarnak and Zubrilina prove the conjectured asymptotic for certain
arithmetic graph families through uniform Plancherel estimates for
growing Hecke polynomials \cite{SarnakZubrilina2024}.
See also Sarnak's Penn State lectures \cite{Sarnak2025Lectures}.

Spectral measures can be compared through their integrals against test functions.
PageRank supplies a test function with a direct probabilistic meaning.
Originally introduced to rank web pages, it describes the distribution of a random walk with restarts \cite{PageEtAl1999,Gleich2015}.
For personalized PageRank based at a vertex $x$ on graph $X$, an equivalent description is to start a simple random walk at $x$ and stop it at an independent time $T$ with $\Pr(T=n)=(1-\alpha)\alpha^n$, where $n\ge0$ and $0<\alpha<1$.
Write $\pi_{\alpha,x}(y)$ for the probability of stopping at $y$, and let $\mathcal R_X(\alpha)$ be the average of $\pi_{\alpha,x}(x)$ over the starting vertices.
The spectral formula for this return probability is
\begin{equation*}
 \mathcal R_X(\alpha)=\int f_\alpha(\lambda)\,d\mu_X(\lambda),
 \qquad
 f_\alpha(\lambda)=\frac{1-\alpha}{1-\alpha\lambda/(p+1)}.
\end{equation*}
The same test against $\nu_p$ gives the return probability $\mathcal M_p(\alpha)$ on the regular tree.
Thus the comparison of two return probabilities is also a comparison
of the finite spectrum with the tree spectrum, so we study this comparison through PageRank.
Each eigenvalue of $X$ has mass $1/N$ in $\mu_X$.
In particular, the constant eigenvector contributes exactly $1/N$ to $\mathcal R_X(\alpha)$, since $f_\alpha(p+1)=1$.

We study this comparison on the Ramanujan graphs constructed by
Lubotzky, Phillips, and Sarnak \cite{LPS1988}, whose adjacency operator is also a Hecke operator,
so we can use arithmetic methods to estimate the spectral averages.
From now on, fix a prime $p\equiv1\pmod4$.
A prime $q$ is called admissible if $q\equiv1\pmod4$, $\left(\frac pq\right)=-1$, and $q>2\sqrt p$.
The corresponding Lubotzky--Phillips--Sarnak graph $X_q$ is connected, bipartite, vertex-transitive, and $(p+1)$-regular.
Its nontrivial adjacency eigenvalues lie in the support of $\nu_p$, the two trivial eigenvalues are $p+1$ and $-(p+1)$.
Its number of vertices $N_q$ and its girth $g_q$, the length of a shortest cycle, satisfy
\begin{equation*}
 N_q=q(q^2-1),\qquad g_q=4\log_pq+O_p(1),
\end{equation*}
where the girth estimate is due to Biggs and Boshier
\cite{BiggsBoshier1990}.
Vertex transitivity gives
$\mathcal R_{X_q}(\alpha)=\pi_{\alpha,x}(x)$ for every vertex $x$.
We write
\begin{equation*}
 \Delta_q(\alpha)
 =\mathcal R_{X_q}(\alpha)-\mathcal M_p(\alpha)
 =\int f_\alpha\,d(\mu_{X_q}-\nu_p).
\end{equation*}
This quantity is nonnegative.
It measures the additional probability
of returning to the starting vertex due to cycles in the finite graph.

We let the test function vary with the level by taking $\alpha_{q,c}=1-\frac c{\log q}$, for $c\in I=[c_0,c_1]\subset(0,\infty)$.
For sufficiently large $q$, the mean stopping time is $\log q/c-1$, of the same order as the girth.
The tree return probability has order $1/\log q$ in this range.
At this scale, the rate of convergence matters: a relative $o(1)$ error in the variance can still contribute significantly to the normalized return correction.
We estimate the smaller contribution $\Delta_q$ at the scale
$1/N_q$, uniformly for $c\in I$.
The theorem below gives bounds at each level, a sharper bound under
the Generalized Riemann Hypothesis (GRH), and moment estimates over
prime levels.

Write $\Acal$ for the admissible primes and $\Acal(Q)=\Acal\cap(Q,2Q]$.

\begin{theorem}
\label{thm:main-results}
Fix $p$ and $I=[c_0,c_1]\subset(0,\infty)$.
\begin{enumerate}
\item[\textup{(i)}] For every sufficiently large admissible prime $q$, uniformly for $c\in I$,
\begin{equation}
\label{eq:individual-return-bound}
 \frac1{N_q}\ll_{p,I}\Delta_q(\alpha_{q,c})
 \ll_{p,I}\frac{\log\log q}{N_q}.
\end{equation}
\item[\textup{(ii)}] Assume the GRH for $L(s,\chi_D)$ for every negative fundamental discriminant $D$, where $\chi_D(n)=\left(\frac{D}
  {n}\right)$.
    For every sufficiently large admissible prime $q$, uniformly for $c\in I$,
\begin{equation}
\label{eq:grh-return}
 \Delta_q(\alpha_{q,c})
 \asymp_{p,I}\frac1{N_q}.
\end{equation}
\item[\textup{(iii)}] For every fixed real number $a\ge1$, as $Q\to\infty$,
\begin{equation}
\label{eq:main-moments}
 \frac1{|\Acal(Q)|}\sum_{q\in\Acal(Q)}
 \left(N_q\sup_{c\in I}\Delta_q(\alpha_{q,c})\right)^a
 \ll_{p,I,a}1.
\end{equation}
\end{enumerate}
\end{theorem}

Part~\textup{(iii)} gives the following estimate for almost all admissible levels.
A subset $\mathcal S\subset\Acal$ has relative density one if $|\mathcal S\cap[1,Y]|/|\Acal\cap[1,Y]|\to1$ as $Y\to\infty$.

\begin{corollary}
\label{cor:density-one}
Let $w:[3,\infty)\to[1,\infty)$ be any nondecreasing function with $w(t)\to\infty$.
There is a subset $\Acal(w)\subset\Acal$ of relative density one such that, for all sufficiently large $q\in\Acal(w)$,
\begin{equation}
\label{eq:density-one}
 0\le\Delta_q(\alpha_{q,c})
 \le\frac{w(q)}{N_q}
 \qquad(c\in I).
\end{equation}
More quantitatively, for every $a\ge1$ and $H\ge1$,
\begin{equation}
\label{eq:exceptional-levels}
 \#\left\{q\in\Acal(Q):
 N_q\sup_{c\in I}\Delta_q(\alpha_{q,c})>H\right\}
 \ll_{p,I,a}\frac{Q}{H^a\log Q}.
\end{equation}
\end{corollary}

The corollary describes the size of the correction for almost all admissible levels: any prescribed factor tending to infinity may replace the double logarithm outside a set of relative density zero.
The bounds also describe PageRank near the source.
For each fixed distance, the excess over the corresponding tree probability is bounded by a constant times $\Delta_q$.
Consequently, the upper bounds and moment estimates hold uniformly over the vertices at that distance.

The arithmetic of the Lubotzky--Phillips--Sarnak construction enters through its path counts.
After the backtracking steps have been summed, $\Delta_q$ is a positive weighted sum of non-backtracking returns.
The adjacency operator is the spherical Hecke operator at $p$, so
these counts are expressed by Hecke polynomials whose degrees grow with the level \cite{LPSHecke1987}.
Quaternionic factorization turns their traces into representation counts by three squares.
The three-square formula then relates these counts to quadratic Dirichlet $L$-values and local factors.
Short Euler products and an upper bound sieve give the estimates in the theorem, while the Ramanujan bound controls the longer paths.

The use of growing Hecke polynomials also appears in the work of Sarnak and Zubrilina on convergence to the Plancherel measure and its applications to non-backtracking walks on arithmetic Ramanujan graphs
\cite{SarnakZubrilina2024}.
Class number averages occur in Zubrilina's work on murmurations
\cite{Zubrilina2025}, related results at prime powers are obtained by
Kundu and M\"uller \cite{KunduMuller2026} and Hua and Mok
\cite{HuaMok2026}.
In our setting, the three-square counts measure the effect of cycles
on the return probability.
There is also a determinant form of the comparison.
The difference between the normalized logarithmic determinant of the finite graph and its tree integral is given by the Ihara zeta function through Bass's formula \cite{Bass1992}.
Differentiating this identity expresses $\Delta_q$ through the
logarithmic derivative of the Ihara zeta function.
These identities illustrate a broader theme in Sarnak's Penn State lectures: trace formulas relate spectral averages to geometric quantities, and the positivity of the geometric counts gives bounds for these averages \cite{Sarnak2025Lectures}.

The paper is organized as follows.
Section~\ref{sec:paths} develops the PageRank formula, the comparison with the tree measure, and the Hecke interpretation of the return counts.
Section~\ref{sec:euler} estimates the three-square representation numbers by short Euler products, both unconditionally and under the GRH.
Section~\ref{sec:prime-euler-moments} proves the required moment
bounds over prime levels.
Section~\ref{sec:main-proof} proves the main results and their consequences near the source.
Appendix~\ref{app:construction} recalls the Lubotzky--Phillips--Sarnak construction.

\section{PageRank, the tree, and arithmetic returns}\label{sec:paths}

\subsection{Personalization and the spectrum}

Let $X$ be a finite connected simple $(p+1)$-regular graph, and let $N=|V(X)|$.
Its adjacency matrix $A_X$ has entry $1$ when two vertices are joined by an edge and entry $0$ otherwise.
The transition matrix is $P_X=A_X/(p+1)$.
We regard probability vectors as columns. Thus, if the present distribution is $v$, the distribution after one step is $P_Xv$.
The symmetry of $A_X$ allows the same matrix to act on functions by averaging their values over neighboring vertices.

For a vertex $x$, the vector $\delta_x$ has value $1$ at $x$ and $0$ elsewhere.
The distribution after $n$ steps from $x$ is $P_X^n\delta_x$.
Averaging over the independent stopping time gives
$$
 \pi_{\alpha,x}
 =(1-\alpha)\sum_{n\ge0}
 \alpha^nP_X^n\delta_x.
$$
Every term $P_X^n\delta_x$ has total mass one, and the weights $(1-\alpha)\alpha^n$ sum to one.
This also shows that the series converges and defines a probability vector.
Multiplication by $I-\alpha P_X$ cancels consecutive terms in the series.
Write $\mathcal R_X(\alpha)=N^{-1}\sum_x\pi_{\alpha,x}(x)$
for the average return probability. Then
\begin{equation}
\label{eq:pagerank-resolvent}
 \pi_{\alpha,x}=(1-\alpha)(I-\alpha P_X)^{-1}\delta_x,
 \qquad
 \mathcal R_X(\alpha)=\frac{1-\alpha}{N}
 \Tr(I-\alpha P_X)^{-1}.
\end{equation}
Here $\Tr$ is the sum of the diagonal entries of a matrix.
The second formula follows by averaging $\pi_{\alpha,x}(x)$ over all starting vertices $x$.
On a vertex-transitive graph, an automorphism carrying one starting vertex to another preserves the walk and its stopping rule.
The return probabilities are therefore equal at all vertices.

More generally, let $v(x)\ge0$ and $\sum_xv(x)=1$.
Choose the starting vertex with probabilities $v(x)$.
Conditioning on that choice gives
\begin{equation*}
 \pi_{\alpha,v}=(1-\alpha)(I-\alpha P_X)^{-1}v
       =\sum_{x\in V(X)}v(x)\pi_{\alpha,x}.
\end{equation*}
This vector satisfies $\pi_{\alpha,v}=(1-\alpha)v+\alpha P_X\pi_{\alpha,v}$.
The two terms describe a restart with distribution $v$ and a random walk step, respectively.
Hence $\pi_{\alpha,v}$ is also the stationary distribution of the walk with these restarts \cite{Gleich2015}.
If $v$ is uniform, then $P_Xv=v$, and the formula gives $\pi_{\alpha,v}=v$.

We can also compute PageRank one eigenvector at a time.
Since $A_X$ is real and symmetric, it has an orthonormal basis of eigenvectors $e_1,\ldots,e_N$, with real eigenvalues $\lambda_1,\ldots,\lambda_N$.
Writing $v=\sum_{j=1}^Na_je_j$ gives $\pi_{\alpha,v}
 =\sum_{j=1}^N \frac{1-\alpha}
{1-\alpha\lambda_j/(p+1)}\,a_je_j$.
Thus each coefficient is multiplied by the displayed factor.
Constant functions have eigenvalue $p+1$, and their coefficients are unchanged.
On a bipartite graph, the function taking values $1$ on one part and $-1$ on the other has eigenvalue $-(p+1)$.
Its coefficient is multiplied by $(1-\alpha)/(1+\alpha)$.

Suppose now that $X$ is Ramanujan, and let $\rho_p=2\sqrt p/(p+1)<1$.
All remaining eigenvalues satisfy $|\lambda_j|/(p+1)\le\rho_p$.
For a vector $f$, let $E_{\mathrm{nt}}f$ be the sum of its eigenvector terms with eigenvalues other than $\pm(p+1)$.
For each of these terms, the multiplier is at most $(1-\alpha)/(1-\alpha\rho_p)$.
Squaring the coefficients and summing therefore proves
\begin{equation*}
 \|E_{\mathrm{nt}}\pi_{\alpha,v}\|_2
 \le \frac{1-\alpha}{1-\alpha\rho_p}
             \|E_{\mathrm{nt}}v\|_2,
\end{equation*}
where $\|f\|_2^2=\sum_{y\in V(X)}|f(y)|^2$.
The bound is independent of $N$.
For $\alpha=\alpha_{q,c}$, its multiplier is at most $c_1/((1-\rho_p)\log q)$, uniformly for $c\in I$.

\subsection{The regular tree}
A tree is a connected graph with no cycles.
In the infinite $(p+1)$-regular tree, fix a root $o$.
The root has $p+1$ neighbors.
Every other vertex has one neighbor closer to $o$ and $p$ neighbors farther from $o$.
There is exactly one path without repeated vertices between any two vertices.

We first count paths on a finite $(p+1)$-regular graph $X$.
A path is called non-backtracking if it never traverses an edge and then immediately traverses the same edge in reverse.
Vertices may be repeated.
For a closed path, this condition applies to consecutive edges in the linear list, the last and first edges are not compared.
Let $b_n(x,y)$ be the number of these paths of length $n$ from $x$ to $y$, and write $B_n=(b_n(x,y))_{x,y}$.
Reversing a path shows that $B_n$ is symmetric.

The following change of variable will express PageRank in terms of these path counts.
For $0<\alpha<1$, choose $u\in(0,p^{-1})$ so that
\begin{equation}
\label{eq:alpha-u}
 \alpha=\frac{(p+1)u}{1+pu^2},
 \qquad D_X(u)=I-uA_X+pu^2I.
\end{equation}
The function on the right increases from $0$ to $1$ on this interval: its derivative is $(p+1)(1-pu^2)/(1+pu^2)^2>0$.
Thus $u$ exists and is unique.
This choice gives $I-\alpha P_X=D_X(u)/(1+pu^2)$.

\begin{lemma}
\label{lem:path-series}
For $0<\alpha<1$ and $u$ as in \eqref{eq:alpha-u},
\begin{equation}\label{eq:path-series}
 \pi_{\alpha,x}(y)=\mathcal M_p(\alpha)
       \sum_{n\ge0}b_n(x,y)u^n,
 \qquad
 \mathcal M_p(\alpha)
       =(1-\alpha)\frac{1+pu^2}{1-u^2}.
\end{equation}
On the infinite regular tree, the corresponding probability at distance $r$ from the source is $\mathcal M_p(\alpha)u^r$.
\end{lemma}

\begin{proof}
There is one path of length zero from a vertex to itself, so $B_0=I$, and $B_1=A_X$.
Among all paths of length two, the forbidden ones go from a vertex to a neighbor and immediately back.
There are $p+1$ of them at each vertex. Hence $B_2=A_X^2-(p+1)I$.

For $n\ge2$, the entry $(A_XB_n)(x,y)$ counts paths obtained by adding an edge $x\to z$ before a non-backtracking path from $z$ to $y$.
An immediate reversal can occur only at the new join.
Such a path starts $x\to z\to x$.
Deleting these two edges leaves a non-backtracking path of length
$n-1$ from $x$ to $y$.
For each remaining path, $z$ can be any neighbor of $x$ except its first neighbor.
There are exactly $p$ choices.
Subtracting the forbidden paths gives
$$
 B_{n+1}=A_XB_n-pB_{n-1}\qquad(n\ge2).
$$

Let $F(u)=\sum_{n\ge0}B_nu^n$.
In the product $D_X(u)F(u)$, the constant coefficient is $I$, the coefficient of $u$ is zero, and the coefficient of $u^2$ is $-I$.
The recurrence makes every later coefficient zero.
Therefore
\begin{equation}
\label{eq:nb-generating}
 \sum_{n\ge0}B_nu^n=(1-u^2)D_X(u)^{-1}.
\end{equation}
These manipulations are valid for $u<p^{-1}$.
Indeed, the number of non-backtracking paths of length $n\ge1$ from a fixed vertex is $(p+1)p^{n-1}$: there are $p+1$ choices for the first step and $p$ for every later step.
The resulting geometric series converges.
Combining \eqref{eq:nb-generating} with \eqref{eq:pagerank-resolvent} and \eqref{eq:alpha-u} proves the finite-graph formula.

On the tree, let $f(y)=\mathcal M_p(\alpha)u^{d(o,y)}$.
There are $(p+1)p^{r-1}$ vertices at distance $r\ge1$, so
$$
 \sum_yf(y)=\mathcal M_p(\alpha)
\left(1+\frac{(p+1)u}{1-pu}\right)
=1.
$$
For the last equality, \eqref{eq:alpha-u} gives $\mathcal M_p(\alpha)=(1-pu)/(1+u)$.
At positive distance, the neighbor counts and the identity $u^r=\alpha(u^{r-1}+pu^{r+1})/(p+1)$ show that $f=\alpha P f$.
At the root, the required equation is
$\mathcal M_p=(1-\alpha)+\alpha\mathcal M_p u$, which follows from the same formulas.
Hence $f=(1-\alpha)\delta_o+\alpha P f$ on the whole tree.
The stopped-walk distribution satisfies this equation as well.
Two probability solutions agree: a random walk step cannot increase the sum of the absolute values of their difference, and iterating the equation bounds that sum by $2\alpha^n$ for every $n$.
Thus $f$ is the stopped-walk distribution on the tree.
Its value at the root is $\mathcal M_p(\alpha)$.
\end{proof}

One can see directly how the tree covers $X$.
Fix $x\in V(X)$ and make a new vertex for every non-backtracking path starting at $x$, including the empty path.
Join a path to each path obtained by adding one allowed step.
The empty path is the root.
Every other path has one predecessor and $p$ extensions, so this graph is the $(p+1)$-regular tree.
Sending each path to its endpoint in $X$ maps the neighbors of each tree vertex bijectively to the neighbors of its image.
Thus every walk on $X$ starting at $x$ lifts uniquely to a walk on the tree starting at the root.
The PageRank mass at $y$ is the sum of the tree masses over all vertices mapping to $y$.
A shortest path of length $r=d(x,y)$ supplies one such vertex.
It contributes $\mathcal M_p(\alpha)u^r$.
All other contributions are nonnegative.
For $y=x$, they give the excess return probability.

We finish by recording the size of $u$ in the range used in the main theorem.
Let $t=pu$.
Substitution into \eqref{eq:alpha-u} and \eqref{eq:path-series} gives
\begin{equation}
\label{eq:critical-relations}
 1-\alpha=\frac{(1-t)(1-t/p)}{1+t^2/p},
 \qquad \mathcal M_p(\alpha)=\frac{1-t}{1+t/p}.
\end{equation}
Write $\varepsilon=1-\alpha$ and $s=1-t$.
The factor $(1-t/p)/(1+t^2/p)$ is bounded above and below by positive constants for $0\le t\le1$, so $s=O_p(\varepsilon)$.
Expanding this factor at $t=1$ then gives
$$
 \varepsilon=\frac{p-1}{p+1}s+O_p(s^2),
 \qquad
 s=\frac{p+1}{p-1}
 \varepsilon+O_p(\varepsilon^2).
$$
With $\varepsilon=c/\log q$, we obtain
\begin{equation}
\label{eq:critical-t}
 t=1-\frac{\tau(c)}{\log q}+O_{p,I}((\log q)^{-2}),
 \qquad \tau(c)=\frac{p+1}{p-1}c.
\end{equation}
The error is uniform for $c\in I$, since $c_0\le c\le c_1$.
In particular, $1-t\asymp_{p,I}1/\log q$ and $\mathcal M_p(\alpha_{q,c})\asymp_{p,I}1/\log q$.

\subsection{The McKay measure and the Ihara zeta function}
\label{sec:spectral}

Let $\mu_X=N^{-1} \sum_{j=1}^N\delta_{\lambda_j}$ be the probability measure that puts mass $1/N$ at each adjacency eigenvalue, counted with multiplicity.
The spectral measure of the regular tree at its root is the McKay measure
\begin{equation*}
 d\nu_p(\lambda)=
 \frac{(p+1)\sqrt{4p-\lambda^2}}
{2\pi((p+1)^2-\lambda^2)}\,d\lambda
 \qquad (|\lambda|\le2\sqrt p),
\end{equation*}
with density zero elsewhere \cite{McKay1981,Sarnak2025Lectures}.
Thus $\int\lambda^n\,d\nu_p(\lambda)$ counts walks of length $n$ from the root back to itself.

For $f_\alpha(\lambda) =(1-\alpha)/(1-\alpha\lambda/(p+1))$, expanding the resolvent in powers of $\alpha$ gives
\begin{equation}
\label{eq:mckay-pagerank}
 \mathcal R_X(\alpha)=\int f_\alpha(\lambda)\,d\mu_X(\lambda),
 \qquad
 \mathcal M_p(\alpha)=\int f_\alpha(\lambda)\,d\nu_p(\lambda).
\end{equation}

For $X=X_q$, the excess $\Delta_q (\alpha)$ is therefore the difference between these two spectral averages.
Theorem~\ref{thm:main-results} estimates this difference for the functions $f_{\alpha_{q,c}}$ as $q$ grows.

There is also an exact relation with a graph zeta function.
A prime cycle is an oriented closed path with no immediate reversal, including at the join between its last and first edges, which is not a proper power of a shorter such path.
Identify cycles under cyclic shifts, with the two orientations kept separate.
The Ihara zeta function and Bass's determinant formula \cite{Bass1992} are
\begin{equation}
\label{eq:ihara-bass}
 Z_X(u)=\prod_{[C]}(1-u^{|C|})^{-1},
 \qquad
 Z_X(u)^{-1}=(1-u^2)^{N(p-1)/2}\det D_X(u).
\end{equation}
Here the product runs over prime cycles and converges for $0<u<p^{-1}$.

\begin{proposition}
\label{prop:spectral-determinant}
For $0<\alpha<1$, let
$$
 F_X(\alpha)=\frac1N\log\det(I-\alpha P_X),
 \qquad
 F_T(\alpha)=
 \int\log\left(1-
 \frac{\alpha\lambda}{p+1}\right)
  d\nu_p(\lambda).
$$
If $u$ is given by \eqref{eq:alpha-u}, then
\begin{equation}
\label{eq:determinant-mckay}
 F_X(\alpha)-F_T(\alpha)
 =-\frac1N\log Z_X(u)
\end{equation}
and
\begin{equation}
\label{eq:ihara-excess}
 \mathcal R_X(\alpha)-\mathcal M_p(\alpha)
 =\frac{(1-\alpha)
 (1+pu^2)u}{N(1-pu^2)}
\frac{Z_X'(u)}{Z_X(u)}.
\end{equation}
\end{proposition}

\begin{proof}
Differentiating the finite eigenvalue sum and the tree integral in \eqref{eq:mckay-pagerank} gives
$$
 \mathcal R_X(\alpha)=(1-\alpha)(1-\alpha F_X'(\alpha)),
\qquad
 \mathcal M_p(\alpha)=(1-\alpha)(1-\alpha F_T'(\alpha)).
$$
For the integral, differentiation is valid since $|\lambda|\le2\sqrt p<p+1$.
Using \eqref{eq:path-series} and differentiating \eqref{eq:alpha-u}, we obtain
$$
 \frac{d}{du}F_T(\alpha(u))
 =-\frac{(p+1)u(1-pu^2)}
 {(1-u^2)(1+pu^2)}
 =-\frac{2pu}{1+pu^2}
 +\frac{(p-1)u}{1-u^2}.
$$
Since $F_T(0)=0$, integration gives
$$
 F_T(\alpha(u))
 =-\log(1+pu^2)-\frac{p-1}{2}
 \log(1-u^2).
$$
On the finite graph, \eqref{eq:alpha-u} gives
$$
 F_X(\alpha(u))
 =\frac1N\log\det
 D_X(u)-\log(1+pu^2).
$$
Subtracting and applying \eqref{eq:ihara-bass} proves \eqref{eq:determinant-mckay}.
Finally, the two derivative identities above give
$$
 \mathcal R_X(\alpha)-\mathcal M_p(\alpha)
 =-\alpha(1-\alpha)
 \frac{d}{d\alpha}
 (F_X(\alpha)-F_T(\alpha)).
$$
Differentiate \eqref{eq:determinant-mckay} and use
$\alpha\,du/d\alpha=u(1+pu^2)/(1-pu^2)$ to obtain \eqref{eq:ihara-excess}.
\end{proof}

\subsection{Hecke operators and quaternionic counts}
For the rest of the paper, $X=X_q$ is the graph in the introduction.
Write $A_q=A_{X_q}$.
The construction in Appendix \ref{app:construction} realizes it as an arithmetic quotient of the Bruhat--Tits tree.
With the maximal compact subgroup $\PGL_2(\mathbb Z_p)$ assigned volume one, $A_q$ is the spherical Hecke operator at $p$.
Define its complete Hecke operators by
\begin{equation}
\label{eq:complete-hecke}
 \mathcal T_n
 =\sum_{0\le j\le n/2}B_{n-2j}.
\end{equation}
They satisfy $\mathcal T_0=I$, $\mathcal T_1=A_q$, and
$$
 A_q\mathcal T_n
 =\mathcal T_{n+1}+p\mathcal T_{n-1}\quad(n\ge1),
 \qquad
 \sum_{n\ge0}\mathcal T_nu^n
 =D_{X_q}(u)^{-1}.
$$
These are the usual Hecke relations
\cite[Lemmas~2.2--2.3]{LPSHecke1987}.
On an eigenspace with eigenvalue
$\lambda=\sqrt p(\beta+\beta^{-1})$, the last resolvent acts by
$((1-\sqrt p\beta u)(1-\sqrt p\beta^{-1}u))^{-1}$.
For a nontrivial automorphic constituent, this is its local standard Euler factor under the substitution $p^{-s}=\sqrt p\,u$.
The automorphic realization is given in \cite[Theorem~4.1 and its proof, pp.~418--419]{LPSHecke1987}.

Let $r_3(m)=\#\{(b_1,b_2,b_3)
\in\mathbb Z^3: b_1^2+b_2^2+b_3^2=m\}$ for $m\ge1$.
For positive integers $P,q$, let
\begin{equation}
\label{eq:arithmetic-sums}
 S(P,q)
 =\sum_{1\le h<P/q^2}
 r_3\bigl(h(P-q^2h)\bigr),
\qquad
\Bcal_q(k)=\frac{q^3}{p^{2k}}S(p^k,q).
\end{equation}
The variable $h$ is integral, and an empty sum is zero.

\begin{lemma}
\label{lem:quaternion-count}
For every admissible $q$ and integer $k\ge1$,
\begin{equation}
\label{eq:quaternion-count}
 \mathcal T_{2k}(x,x)
 =1+S(p^k,q),
 \qquad
 b_{2k}(x,x)
 =S(p^k,q)-S(p^{k-1},q).
\end{equation}
In particular, $0\le b_{2k}(x,x)\le S(p^k,q)$.
\end{lemma}

\begin{proof}
A word of length $n$ in the Lubotzky--Phillips--Sarnak generators describes a path of length $n$.
Quaternionic factorization identifies non-backtracking words with primitive
integral quaternions
$$
 a=a_0+a_1i+a_2j+a_3k,
 \qquad
 \operatorname{Nrd}(a)=p^n,
$$
up to an overall sign, see \cite[Lemma~2.3]{LPSHecke1987}.
Here $a_0$ is odd, the other three coordinates are even, and primitive
means that the four coordinates are not all divisible by $p$.
Since $a_0$ is odd, it is nonzero.
We choose the unique overall sign for which $a_0>0$.

Every integral quaternion $a$ of norm $p^{2k}$ with these parity conditions
has a unique expression $a=p^j a'$ when $0\le j\le k$, where $a'$ is primitive.
Indeed, $j$ is the largest exponent such that $p^j$ divides all four coordinates.
Dividing by $p^j$ preserves the parity
conditions and the positivity of the real coordinate, and gives
$\operatorname{Nrd}(a')=p^{2k-2j}$.
Thus $a'$ corresponds to a non-backtracking path of length $2k-2j$.
Conversely, multiplying the primitive quaternion of such a path by $p^j$
gives a quaternion of norm $p^{2k}$.
This explains the decomposition
$\mathcal T_{2k}=B_{2k}+B_{2k-2}+\cdots+B_0$ in \eqref{eq:complete-hecke}.

We first count returns at the identity vertex of $\PGL_2(\F_q)$.
Let $\rho_q(a)$ denote the matrix obtained by reducing the quaternion
coordinates modulo $q$ and applying the matrix map in the Lubotzky--Phillips--Sarnak construction.
The endpoint of the associated path is the projective class $[\rho_q(a)]$.
Since the identity class in $\PGL_2(\F_q)$ consists of the nonzero
scalar matrices, the path returns precisely when $\rho_q(a)=cI$ for some $c\in\F_q^\times$.
The matrix is invertible because its determinant is the quaternion norm,
which is a power of $p$, and $q\ne p$.
Multiplication or division by $p$ changes this matrix by a nonzero scalar
modulo $q$, so it preserves the return condition.

It follows that $\mathcal T_{2k}(x,x)$ counts all integral quaternions
of norm $p^{2k}$ with positive real coordinate, the stated parity
conditions, and scalar image modulo $q$.
The scalar quaternion $p^k$ is included, its primitive part is $1$,
which corresponds to the empty path.
Left translation is a graph automorphism, so the same count holds
at every vertex $x$.

Write $P=p^k$.
The quaternion matrix map shows that the scalar condition is equivalent to
$a_1\equiv a_2\equiv a_3\equiv0\pmod q$.
Since these coordinates are also even and $q$ is odd, we may write $a_i=2qb_i$ for $i=1,2,3$.
The norm equation becomes $a_0^2+4q^2(b_1^2+b_2^2+b_3^2)=P^2$,
here $a_0>0$ odd.
There is exactly one scalar quaternion in this set: $a_0=P$ and $b_1=b_2=b_3=0$.
It accounts for the term $1$ in the first identity.

For every remaining quaternion, $0<a_0<P$.
Let $d_-=\frac{P-a_0}{2}$, $d_+=\frac{P+a_0}{2}$.
These are positive integers satisfying
$d_-+d_+=P$, $d_-d_+=q^2(b_1^2+b_2^2+b_3^2)$.
Since $q\nmid P$, the prime $q$ cannot divide both $d_-$ and $d_+$.
Their product is divisible by $q^2$, so exactly one of them is divisible
by $q^2$.
Write that factor as $q^2h$.
The other factor is then $P-q^2h$, and their positivity gives $0<h<\frac{P}{q^2}$.
Moreover,
$$
 a_0=|d_+-d_-|=|P-2q^2h|,
 \qquad
 b_1^2+b_2^2+b_3^2=h(P-q^2h).
$$
The integer $h$ is uniquely determined, since exactly one of
$d_-$ and $d_+$ is divisible by $q^2$.

Conversely, take an integer $h$ with $0<h<P/q^2$ and an integer triple
satisfying $b_1^2+b_2^2+b_3^2 =h(P-q^2h)$.
Let $a_0=|P-2q^2h|$, $a_i=2qb_i$ with $i=1,2,3$.
The number $a_0$ is odd and nonzero, since $P$ is odd.
The bounds on $h$ give $a_0<P$, and direct substitution gives $a_0^2+a_1^2+a_2^2+a_3^2=P^2$.
Thus these coordinates define a quaternion with the required positivity, parity, norm, and scalar image modulo $q$.
The uniqueness of $h$ shows that each quaternion is counted once.
Consequently,
$$
 \mathcal T_{2k}(x,x)
 =1+\sum_{1\le h<P/q^2}r_3\bigl(h(P-q^2h)\bigr)
 =1+S(P,q).
$$
This proves the first identity.

Finally, \eqref{eq:complete-hecke} gives $B_{2k}=\mathcal T_{2k}-\mathcal T_{2k-2}$.
Taking diagonal entries and applying the first identity yields
$$
 b_{2k}(x,x)
 =\bigl(1+S(p^k,q)\bigr)
  -\bigl(1+S(p^{k-1},q)\bigr)
 =S(p^k,q)-S(p^{k-1},q).
$$
The lower bound follows because $b_{2k}(x,x)$ counts paths.
The upper bound follows from $S(p^{k-1},q)\ge0$.
\end{proof}

The graph is bipartite, so odd return counts vanish.
Lemma~\ref{lem:path-series} therefore gives the convergent positive series
\begin{equation}
\label{eq:return-positive-series}
 \Delta_q(\alpha)
 =\mathcal M_p(\alpha)
 \sum_{k\ge1}b_{2k}(x,x)u^{2k}
 =\mathcal M_p(\alpha)(1-u^2)
 \sum_{k\ge1}S(p^k,q)u^{2k}.
\end{equation}
For the second equality, sum \eqref{eq:quaternion-count}.
The series converges because $S(p^k,q)
=\sum_{j=1}^k b_{2j}(x,x) =O_p(p^{2k})$.

\subsection{The long paths}
The Ramanujan bound controls the part of \eqref{eq:return-positive-series} beyond
\begin{equation}
\label{eq:cutoff-index}
 K_q=\lfloor4\log_pq\rfloor.
\end{equation}

\begin{lemma}
\label{lem:return-reduction}
Uniformly for sufficiently large admissible $q$,
\begin{equation}
\label{eq:return-reduction}
 N_q\sup_{c\in I}\Delta_q(\alpha_{q,c})
 \ll_{p,I}1+\frac1{\log q}\sum_{k\le K_q}\Bcal_q(k).
\end{equation}
\end{lemma}

\begin{proof}
The path recurrence expresses $B_n=Q_n(A_q)$, with $Q_0=1$, $Q_1(\lambda)=\lambda$, and, for $n\ge2$,
$$
 Q_n(\lambda)=p^{n/2}
 U_n\left(\frac{\lambda}{2\sqrt p}\right)
 -p^{(n-2)/2}U_{n-2}
 \left(\frac{\lambda}{2\sqrt p}\right).
$$
Here $U_n$ is the Chebyshev polynomial of the second kind, characterized by $U_0=1$, $U_1(v)=2v$, and $U_{n+1}(v)=2vU_n(v)-U_{n-1}(v)$.
It satisfies $|U_n(v)|\le n+1$ for $|v|\le1$.
Each of the two trivial spectral projections has diagonal entries $1/N_q$.
The other eigenvalues lie in $[-2\sqrt p,2\sqrt p]$.
The spectral theorem thus gives
\begin{equation}
\label{eq:spectral-return}
 b_{2k}(x,x)
 =\frac{2(p+1)p^{2k-1}}{N_q}
 +O_p((k+1)p^k).
\end{equation}
The sum of the nontrivial spectral projections has diagonal at most one.

For $\alpha=\alpha_{q,c}$, write $M=\mathcal M_p(\alpha)$ and $t=pu$.
Equations \eqref{eq:critical-relations} and
\eqref{eq:critical-t} give
\begin{align*}
 M\sum_{k>K_q}b_{2k}(x,x)u^{2k}
 &\ll_p\frac{M}{N_q}\sum_{k>K_q}t^{2k}
 +M\sum_{k>K_q}(k+1)(t^2/p)^k
 \\
 &\ll_{p,I}\frac1{N_q}+M(K_q+1)p^{-K_q}
 \ll_{p,I}q^{-3}.
\end{align*}
For $k\le K_q$, use $b_{2k}\le S(p^k,q)$ and $t\le1$ to obtain
$$
 M\sum_{k\le K_q}
 b_{2k}(x,x)u^{2k}
 \le q^{-3}M
 \sum_{k\le K_q}\Bcal_q(k).
$$
Since $M\ll_{p,I}1/\log q$ and $N_q\asymp q^3$, the result follows.
\end{proof}

\section{Three squares and short Euler products}\label{sec:euler}

The three-square formula reduces the return count to quadratic $L$-values and local factors.
We first bound them by a single Euler product and then shorten this product for use in the sieve.

For a positive integer $m$, define a completely multiplicative function
$\psi_m$ by its values at primes:
\begin{equation*}
 \psi_m(\ell)=
 \begin{cases}
  1,&\ell\mid2m,\\
  \left(\dfrac{-m}
  {\ell}\right),
  &\ell\nmid2m.
 \end{cases}
\end{equation*}
Thus $\psi_m(\ell)$ is always $1$ or $-1$.
For $\Re s>1$ and $z\ge2$, let
\begin{equation*}
 L_{\psi_m}(s)
 =\prod_\ell(1-\psi_m(\ell)\ell^{-s})^{-1},
 \qquad
 E_z(m)=\prod_{\ell\le z}(1-\psi_m(\ell)/\ell)^{-1}.
\end{equation*}
Let $D$ be the fundamental discriminant of $\mathbb Q(\sqrt{-m})$, and let $\chi_D(n)=\left(\frac{D}
  {n}\right)$ be its primitive quadratic Dirichlet character.
Symbols with a general denominator denote Kronecker symbols.
Away from the primes dividing $2m$, the functions $\psi_m$ and $\chi_D$ agree.
Consequently,
\begin{equation}\label{eq:finite-euler-modification}
 L_{\psi_m}(s)=L(s,\chi_D)
 \prod_{\ell\mid2m}
 \frac{1-\chi_D(\ell)\ell^{-s}}{1-\ell^{-s}}.
\end{equation}
This identity defines a positive finite value at $s=1$.

\begin{lemma}\label{lem:class-number-majorant}
For every positive integer $m$, $r_3(m)\ll\sqrt m\,L_{\psi_m}(1)$.
\end{lemma}

\begin{proof}
Write $\chi_{-4m}(n)=\left(\frac{-4m}
  {n}\right)$, allowing this character to be imprimitive.
The three-square formula \cite[Theorem~1.4]{BerkovichJagy2012} gives
$$
 r_3(m)=\frac{16}{\pi}\sqrt m\,\eta_2(m)
 L(1,\chi_{-4m})\prod_{\substack{\ell^2\mid m\\\ell\text{ odd}}}
 \mathcal P_\ell(m),
$$
where $0\le\eta_2(m)\le3/2$.
Write $v_\ell(m)$ for the exponent of $\ell$ in $m$.
If $b=\lfloor v_\ell(m)/2\rfloor\ge1$, the odd local factor has the form
$$
 \mathcal P_\ell(m)
 =\sum_{j=0}^{b-1}\ell^{-j}
 +\frac{\ell^{-b}}
 {1-\varepsilon_\ell/\ell},
 \qquad
 \varepsilon_\ell=
 \left(\frac{-m\ell^{-2b}}
 {\ell}\right)\in\{-1,0,1\}.
$$
It is at most $(1-1/\ell)^{-1}$.
Every prime dividing $2m$ has value zero under $\chi_{-4m}$, so
$$
 L_{\psi_m}(1)
 =L(1,\chi_{-4m})
 \prod_{\ell\mid2m}(1-1/\ell)^{-1}.
$$
This contains all the required local upper bounds and proves the lemma.
\end{proof}

\begin{lemma}
\label{lem:short-euler-upper}
If $V\ge\max(16,4m)$ and $2\le z\le V$, then $r_3(m)\ll\sqrt m\,\frac{\log V}{\log z}\,E_z(m)$.
\end{lemma}
\begin{proof}
We first compare values just to the right of $1$.
Since $D<0$, the primitive character $\chi_D$ is odd.
The Hadamard product and functional equation give, for $1\le\sigma\le2$,
\begin{equation*}
-\frac{L'}{L}(\sigma,\chi_D)
 =\frac12\log\frac{|D|}{\pi}
 +\frac12\frac{\Gamma'}{\Gamma}
 \left(\frac{\sigma+1}{2}\right)
 -\sum_\rho\Re\frac1{\sigma-\rho}
 \le\frac12\log|D|+O(1).
\end{equation*}
The finite modification in \eqref{eq:finite-euler-modification} increases the negative logarithmic derivative by at most $2\sum_{\ell\mid2m}
\frac{\log\ell}{\ell-1}\le2\log(2m)$.
As $|D|\le4m$, it follows that $-\frac{L_{\psi_m}'}
{L_{\psi_m}}(\sigma)\le C\log V$ when
$1\le\sigma\le2$.
Here and below in this proof, $C$ denotes an absolute constant.
Let $\sigma_0=1+1/\log V$ and $\sigma_1=1+1/\log z$.
Integration on $[1,\sigma_0]$ yields
\begin{equation}\label{eq:first-shift}
 L_{\psi_m}(1)\ll L_{\psi_m}(\sigma_0).
\end{equation}

For $\sigma>1$, the absolutely convergent Dirichlet series gives, with $\Lambda$ denoting the von Mangoldt function, $-\frac{L_{\psi_m}'}
{L_{\psi_m}}(\sigma)
 =\sum_{n\ge1}
 \frac{\Lambda(n)\psi_m(n)}{n^\sigma}
 \le-\frac{\zeta'}{\zeta}(\sigma)$.
Since $\sigma_0\le\sigma_1$, integration gives
\begin{equation}
\label{eq:zeta-comparison}
 \frac{L_{\psi_m}(\sigma_0)}
 {L_{\psi_m}(\sigma_1)}
 \le\frac{\zeta(\sigma_0)}
 {\zeta(\sigma_1)}
 \ll\frac{\log V}{\log z}.
\end{equation}
Chebyshev's bound for primes and partial summation show that
$\sum_{\ell>z}\ell^{-1-1/\log z}\ll1$, $\frac1{\log z}\sum_{\ell\le z}\frac{\log\ell}{\ell}\ll 1$.
The first estimate bounds the Euler product over $\ell>z$ at $\sigma_1$.
The second bounds the change in the remaining factors between $\sigma_1$ and $1$.
Thus $L_{\psi_m}(\sigma_1)\ll E_z(m)$.
Now combine \eqref{eq:first-shift}, \eqref{eq:zeta-comparison}, and
Lemma~\ref{lem:class-number-majorant}.
\end{proof}

Taking $z=2$ and $V=\max(16,4m)$ gives the useful uniform bound $r_3(m)\ll\sqrt m\log(2m)$.

\begin{lemma}
\label{lem:grh-short-euler}
Let $Y$ be sufficiently large.
Assume the GRH for $L(s,\chi_D)$ for every negative fundamental discriminant $D$ with $|D|\le Y$.
Let $z=(\log Y)^8$. For every integer $1\le m\le Y/4$,
\begin{equation}
\label{eq:grh-short-euler}
 r_3(m)\ll\sqrt m\,E_z(m),
\end{equation}
with an absolute implied constant.
\end{lemma}

\begin{proof}
The fundamental discriminant $D$ of $\mathbb Q(\sqrt{-m})$ satisfies
$|D|\le4m\le Y$.
The short Euler product estimate of Granville and Soundararajan
\cite[Lemma~2.1]{GranvilleSoundararajan2003}, applied with $s=1$ and $\sigma_0=1/2$, gives
\begin{equation*}
 \log L(1,\chi_D)
 =\sum_{2\le n\le z}\frac{\Lambda(n)\chi_D(n)}{n\log n}
 +O\left(\frac{(\log Y)(\log z)^2}{\sqrt z}\right).
\end{equation*}
The error is bounded uniformly in $D$.
Replacing the sum by $\log\prod_{\ell\le z}(1-\chi_D(\ell)/\ell)^{-1}$
costs $O(z^{-1/2})$.
Indeed, omitted powers of primes
$\ell>\sqrt z$ contribute $O(\sum_{n>\sqrt z}n^{-2})$,
and the omitted powers of each prime $\ell\le\sqrt z$ contribute $O(1/z)$.
Exponentiating gives
\begin{equation*}
 L(1,\chi_D)\ll
 \prod_{\ell\le z}
 (1-\chi_D(\ell)/\ell)^{-1}.
\end{equation*}

Now use \eqref{eq:finite-euler-modification} at $s=1$.
The factors at primes $\ell\le z$ become exactly those of $E_z(m)$.
The remaining correction factors have logarithm at most $O\left(\sum_{\substack{
 \ell\mid2m\\\ell>z}}\frac1\ell\right)
 =O\left(\frac{\log(2m)}{z\log z}\right)
 =O(1)$.
Thus $L_{\psi_m}(1)\ll E_z(m)$.
Lemma~\ref{lem:class-number-majorant} proves \eqref{eq:grh-short-euler}.
\end{proof}

\section{Prime averages of short Euler products}
\label{sec:prime-euler-moments}

We now average over the prime level. We divide each local factor of
$E_z^j$ by its largest possible value and express the ratios as probabilities of avoiding residue classes.
An upper bound sieve then gives a bound uniform in both $P$ and $h$.
We apply it to each exponent in the return expansion.

Products indexed by $\ell$ are over primes.
We write $\varphi$ for Euler's totient function and $\omega(d)$ for the number of distinct prime divisors of $d$.

\subsection{A sieve lemma and divisor estimates}
We use the following consequence of the beta-sieve.
Let $\beta_\kappa$ be its sifting limit in dimension $\kappa$, in particular, $\beta_1=2$.

\begin{lemma}[Upper bound sieve
{\cite[Theorems~11.12--11.13]
{FriedlanderIwaniec}}]
\label{lem:upper-bound-sieve}
Fix $\kappa\ge1$, $L\ge1$, and
$s\ge\max\{2,\beta_\kappa-1\}$.
Let $g$ be a multiplicative function with $0\le g(\ell)<1$ and
\begin{equation}
\label{eq:beta-dimension}
 \prod_{w\le\ell<v}(1-g(\ell))^{-1}
 \le
 \left(\frac{\log v}{\log w}\right)^\kappa
 \left(1+\frac{L}{\log w}\right)
 \qquad(2\le w<v).
\end{equation}
For $y\ge2$, let $\mathcal P(y)=\prod_{\ell<y}\ell$, $V(y)=\prod_{\ell<y}(1-g(\ell))$, and $D=y^s$.
Let $\mathcal A=(a_n)$ be a finitely supported sequence of nonnegative real numbers.
Suppose that, for every $d\mid\mathcal P(y)$ with $d\le D$, $|\mathcal A_d|:=\sum_{d\mid n}a_n=Xg(d)+r_d$ with $X\ge0$.
Then
\begin{equation*}
 S(\mathcal A,y):=
 \sum_{\substack{n\ge1\\(n,\mathcal P(y))=1}}a_n
 \ll_{\kappa,L,s}
 X V(y)+
 \sum_{\substack{d\le D\\d\mid\mathcal P(y)}}|r_d|.
\end{equation*}
\end{lemma}

\begin{proof}
For fixed $s$, the main term in the cited upper bound is $O_{\kappa,L,s}(XV(y))$.
The beta-sieve weights have absolute value at most one and are supported on $d\le D$, which gives the stated remainder.
\end{proof}

We also need two elementary divisor estimates.
Fix an integer $\kappa\ge1$.
If $\tau_\kappa(n)$ counts the ordered factorizations of $n$ into $\kappa$ positive integers, then $\kappa^{\omega(n)} \le\tau_\kappa(n)$ and
\begin{equation}
\label{eq:moment-divisor-sum}
 \sum_{n\le y}\kappa^{\omega(n)}
 \le\sum_{n\le y}\tau_\kappa(n)
 \ll_\kappa y(\log(2y))^{\kappa-1}.
\end{equation}
The last estimate follows by summing
$1/(n_1\cdots n_{\kappa-1})$ over the first $\kappa-1$ factors.

Define a nonnegative multiplicative function $b_\kappa$, supported
on square-free integers, by $b_\kappa(\ell)=(1-1/\ell)^{-\kappa}-1$.
Then $(n/\varphi(n))^\kappa=\sum_{e\mid n}b_\kappa(e)$ and
$b_\kappa(\ell)\ll_\kappa1/\ell$. Consequently,
\begin{equation}
\label{eq:moment-totient-average}
 \sum_{n<y}\left(\frac n{\varphi(n)}\right)^\kappa
 \le y\sum_{e\ge1}\frac{b_\kappa(e)}e
 \ll_\kappa y\qquad(y>1).
\end{equation}
The last series converges because its Euler factors are $1+O_\kappa(\ell^{-2})$.

\subsection{An Euler product average}
\begin{lemma}
\label{lem:prime-euler-moments}
Fix a positive integer $j$.
Let $P,h$ be positive integers, let $z\ge2$, and let $J$ be a bounded interval in $(z,\infty)$ on which $m_q=h(P-q^2h)$ is positive.
Then
\begin{equation}
\label{eq:prime-euler-moments}
 \sum_{\substack{q\in J\\q\text{ prime}}}E_z(m_q)^j
 \ll_j
 \left(\frac{hP}{\varphi(hP)}\right)^{j+1}
 \frac{|J|}{\log z}
 +z^2(\log(2z))^{3j}.
\end{equation}
Here $|J|$ is the length of $J$.
The constant is independent of $h$, $P$, $J$, and $z$.
\end{lemma}

\begin{proof}
After normalization, the Euler products become probabilities of avoiding random residue classes.
Excluding multiples of small primes
gives a term $1/\ell$ in the local density.
This gives sieve dimension $\kappa=j+1$.

Let $\kappa=j+1$, $M=2hP$, $F(a)=h^2a^2-hP$, and $\mathfrak M(z)=\prod_{\ell\le z}(1-1/\ell)^{-1}$.
For each odd prime $\ell\nmid M$, define
\begin{equation*}
 \Omega_\ell=
 \left\{a\bmod\ell:
 \left(\frac{F(a)}\ell\right)=-1\right\}.
\end{equation*}
The elementary identity
$\sum_{a\bmod\ell}(\frac{a^2-b}{\ell})=-1$ for $b\ne0\bmod\ell$ follows by counting the solutions of $(a-v)(a+v)=b$.
It gives
\begin{equation}
\label{eq:moment-local-counts}
 |\Omega_\ell|
 =\frac{\ell-\left(\frac{hP}{\ell}\right)}2,
 \qquad
 \mathbf 1_{0\in\Omega_\ell}
 =\frac{1-\left(\frac{-hP}{\ell}\right)}2.
\end{equation}
At a prime dividing $h$, the polynomial $F$ vanishes identically.
At an odd prime dividing $P$ it is the square $h^2a^2$.
Thus the Euler factors at primes dividing $M$ all have the value $(1-1/\ell)^{-1}$.

For $\ell\nmid M$, choose independent random variables $\epsilon_\ell\in\{0,1\}$ with
\begin{equation}
\label{eq:moment-activation}
 \Pr(\epsilon_\ell=1)=a_{\ell,j}
 =1-\left(\frac{\ell-1}{\ell+1}\right)^j.
\end{equation}
When $\epsilon_\ell=1$, exclude the classes in $\Omega_\ell$.
Also exclude the zero class at every prime.
More precisely, let
\begin{equation*}
 \Gamma_\ell=
 \begin{cases}
  \{0\}\cup\Omega_\ell,&\ell\nmid M,\ \epsilon_\ell=1,\\
  \{0\},&\text{otherwise}.
 \end{cases}
\end{equation*}
For a prime $q>z$, the probability that $q$ avoids all the sets $\Gamma_\ell$ with $\ell\le z$ is $E_z(m_q)^j/\mathfrak M(z)^j$.
Indeed, the ratio between the factors with $\psi_m(\ell)=-1$ and $\psi_m(\ell)=1$ is $(\ell-1)/(\ell+1)$.
Consequently,
\begin{equation}
\label{eq:moment-survival-majorant}
 \sum_{\substack{q\in J\\q\text{ prime}}}E_z(m_q)^j
 \le \mathfrak M(z)^j
 \sum_{a\in J\cap\mathbb Z}
 \mathbb E\,\mathbf 1_{a\notin\Gamma_\ell\ (\ell\le z)}.
\end{equation}

Write $g(\ell)=\mathbb E|\Gamma_\ell|/\ell$.
For $\ell\nmid M$, equations
\eqref{eq:moment-local-counts} and \eqref{eq:moment-activation} give
\begin{equation}
\label{eq:moment-local-density}
 g(\ell)=\frac1\ell+
 a_{\ell,j}\frac{\ell-1+\delta_\ell}{2\ell},
 \qquad
 \delta_\ell
 =\left(\frac{-hP}\ell\right)
 -\left(\frac{hP}\ell\right)\in\{-2,0,2\}.
\end{equation}
For $\ell\mid M$, one has $g(\ell)=1/\ell$.
The inequality $a_{\ell,j}\le2j/(\ell+1)$ shows that $\ell g(\ell)\le\kappa$.
Moreover, $0<g(\ell)<1$ at every prime.
For a prime outside $M$, there are at most $\ell-1$ nonzero classes in $\Omega_\ell$, and hence
\begin{equation}
\label{eq:moment-small-primes}
 1-g(\ell)\ge
 \left(1-\frac1\ell\right)
 \left(\frac{\ell-1}{\ell+1}\right)^j>0.
\end{equation}
The same lower bound holds at primes dividing $M$.

Choose a fixed $s_j\ge\max\{2,\beta_\kappa+1\}$ and let $y=z^{2/s_j}$, $D=z^2=y^{s_j}$.
We first suppose that $y\ge2$.
The Euler product remains truncated at $z$, and we sieve with the primes $\ell<y$.

For an integer $a\in J$ and a choice of the random sets, let
$n(a)=\prod_{\substack{
\ell<y\\a\bmod\ell\in\Gamma_\ell}}\ell$, $A_n=\sum_{a\in J\cap\mathbb Z}\Pr(n(a)=n)$.
This defines a finite nonnegative sequence $\mathcal A=(A_n)$.
Its sifted sum is
$$
 S(\mathcal A,y)=
 \sum_{a\in J\cap\mathbb Z}
 \Pr(a\bmod\ell\notin\Gamma_\ell\text{ for every }\ell<y).
$$
Since $y\le z$, equation \eqref{eq:moment-survival-majorant} gives
\begin{equation}
\label{eq:moment-sieve-majorant}
 \sum_{\substack{q\in J\\q\text{ prime}}}
 E_z(m_q)^j
 \le
 \mathfrak M(z)^j S(\mathcal A,y).
\end{equation}

Extend $g$ multiplicatively to square-free integers.
For $d\mid\mathcal P(y)$, independence and the Chinese remainder
theorem give
\begin{equation}
\label{eq:moment-crt-average}
 |\mathcal A_d|
 =|J|g(d)+r_d,\qquad
 |r_d|\ll\kappa^{\omega(d)}.
\end{equation}
Indeed, each residue class modulo $d$ contains $|J|/d+O(1)$ integers in $J$, and the expected number of excluded
classes modulo $d$ is $\prod_{\ell\mid d}\ell g(\ell)\le\kappa^{\omega(d)}$.

We check the dimension condition uniformly in $h$ and $P$.
Expanding \eqref{eq:moment-local-density} gives
$g(\ell)=\frac{\kappa}{\ell}+O_j(\ell^{-2})$ for $\ell\nmid M$, and $g(\ell)=\frac1\ell$ for $\ell\mid M$.
Thus, for sufficiently large $\ell$ in terms of $j$, $-\log(1-g(\ell))\le \kappa/\ell+O_j(\ell^{-2})$.
Mertens' estimate for the sum of reciprocal primes verifies \eqref{eq:beta-dimension} with $L=L_j$.
The finitely many smaller primes are controlled uniformly by
\eqref{eq:moment-small-primes}.
At a prime outside $M$ we also have $\frac{1-g(\ell)}{(1-1/\ell)^\kappa}
 =1+O_j(\ell^{-2})$.
The product of these ratios is bounded above and below by positive constants depending only on $j$.
Since $M=2hP$, it follows that
\begin{equation}
\label{eq:moment-sifting-product}
 V(y)
 \asymp_j(\log y)^{-\kappa}
 \prod_{\substack{\ell<y\\\ell\mid M}}
(1-1/\ell)^{1-\kappa}
 \ll_j
 \left(\frac{hP}{\varphi(hP)}\right)^\kappa
 (\log y)^{-\kappa}.
\end{equation}

Apply Lemma \ref{lem:upper-bound-sieve} with $X=|J|$ and use
\eqref{eq:moment-divisor-sum}. Equations
\eqref{eq:moment-crt-average} and \eqref{eq:moment-sifting-product}
give
$$
 S(\mathcal A,y)\ll_j
 \left(\frac{hP}{\varphi(hP)}\right)^\kappa
 \frac{|J|}{(\log y)^\kappa}
 +D(\log(2D))^{\kappa-1}.
$$
Finally, $\mathfrak M(z)\ll\log(2z)$ by Mertens' product estimate
\cite{MontgomeryVaughan2006}, while
$\log y=(2/s_j)\log z$ and $\kappa=j+1$.
Substitution in \eqref{eq:moment-sieve-majorant} yields
$$
 \sum_{\substack{q\in J\\q\text{ prime}}}E_z(m_q)^j
 \ll_j
 \left(\frac{hP}{\varphi(hP)}\right)^{j+1}
 \frac{|J|}{\log z}
 +z^2(\log(2z))^{2j}.
$$
This proves \eqref{eq:prime-euler-moments}.
If $y<2$, then $z$ is bounded in terms of $j$.
The bound $E_z(m)^j\le\mathfrak M(z)^j\ll_j1$ and the estimate
$\#(J\cap\mathbb Z)\le|J|+1$ prove the same assertion.
\end{proof}

\subsection{Moments of the arithmetic sums}
\begin{proposition}
\label{prop:prime-shell-moments}
Fix a prime $p$ and a positive integer $j$. For $Q\ge3$ and every positive integer $k$ with $p^k\le(2Q)^4$, one has
\begin{equation}
\label{eq:prime-shell-moments}
 \sum_{\substack{Q\le q\le2Q\\q\text{ prime}}}
 \mathcal B_q(k)^j\ll_{p,j}\frac Q{\log Q}.
\end{equation}
Here $\mathcal B_q(k)$ is defined by \eqref{eq:arithmetic-sums}.
The constant in \eqref{eq:prime-shell-moments} is independent of $k$ and $Q$.
\end{proposition}

\begin{proof}
Let $P=p^k$ and $H=P/Q^2$.
If $H\le1$, all the sums vanish.
Suppose $H>1$.
We first take $Q$ sufficiently large in terms of $j$ and let
\begin{equation}
\label{eq:moment-sieve-parameters}
 z=Q^{1/8},\qquad V=256Q^6.
\end{equation}
For $Q\le q\le2Q$ and $1\le h<P/q^2$, completing the square gives $0<m_q=h(P-q^2h)\le
 \frac{P^2}{4q^2}\le64Q^6$.
Lemma~\ref{lem:short-euler-upper} gives $r_3(m_q)\ll
 \sqrt{m_q}\frac{\log V}{\log z}E_z(m_q)
 \ll\frac Pq E_z(m_q)$.
The ratio of logarithms is bounded by an absolute constant.
Since $q^2/P\le4/H$, it follows that $\mathcal B_q(k)
 \ll\frac1H\sum_{\substack{1\le h<H\\h<P/q^2}}E_z(m_q)$.
There are at most $H$ terms.
H\"older's inequality therefore yields
\begin{equation}
\label{eq:moment-holder-h}
 \mathcal B_q(k)^j
 \ll_j\frac1H
 \sum_{\substack{1\le h<H\\h<P/q^2}}E_z(m_q)^j.
\end{equation}

For each $h<H$, the permitted values of $q$ form the interval $J_h=[Q,2Q]\cap\bigl(0,\sqrt{P/h}\bigr)$,
which has length at most $Q$.
Lemma~\ref{lem:prime-euler-moments}
and \eqref{eq:moment-sieve-parameters} give
\begin{equation}
\label{eq:moment-fixed-h-bound}
 \sum_{\substack{q\in J_h\\q\text{ prime}}}E_z(m_q)^j
 \ll_j
 \left(\frac{hP}{\varphi(hP)}\right)^{j+1}
 \frac Q{\log Q}
 +Q^{1/4}(\log Q)^{3j}.
\end{equation}
Since $P$ is a power of the fixed prime $p$, $\frac{hP}{\varphi(hP)}
 \le\frac p{p-1}\frac h{\varphi(h)}$.
Equation \eqref{eq:moment-totient-average}, with $\kappa=j+1$, gives $\sum_{1\le h<H}\left(\frac{hP}{\varphi(hP)}\right)^{j+1}
 \ll_{p,j}H$.
Summing \eqref{eq:moment-fixed-h-bound} over $h$, and then using \eqref{eq:moment-holder-h}, proves
\begin{equation*}
 \sum_{\substack{Q\le q\le2Q\\q\text{ prime}}}\mathcal B_q(k)^j
 \ll_{p,j}\frac Q{\log Q}
  +Q^{1/4}(\log Q)^{3j}
 \ll_{p,j}\frac Q{\log Q}.
\end{equation*}
The last step uses that $j$ is fixed.
Increasing the constant covers the bounded range of $Q$ left at the start of the proof.
\end{proof}

\section{Proof of the main results}\label{sec:main-proof}

Throughout this section, let $D_q=N_q\sup_{c\in I}\Delta_q(\alpha_{q,c})$.
\begin{proof}[Proof of Theorem~\ref{thm:main-results}\textup{(i)}]
\emph{The upper bound.} Apply the sieve with $q$ fixed.
For $k\le K_q$, let $P=p^k$ and $H=H_{q,k}=P/q^2$.
The sum in \eqref{eq:arithmetic-sums} vanishes when $H\le1$. For $H>1$, we claim that
\begin{equation}
\label{eq:individual-shell-bound}
 \mathcal B_q(k)\ll_p
 \frac{\log(2qH)}{\log(2H)}
 \ll_p\frac{\log q}{\log(2H)}.
\end{equation}
The last inequality uses $H\le q^2$, which follows from \eqref{eq:cutoff-index}.

Write $m_h=h(P-q^2h)$.
For an odd prime $\ell\nmid P$, the polynomial $q^2h^2-Ph$ has exactly $(\ell-1)/2$ classes with Legendre symbol $-1$.
The character sum identity used in Lemma~\ref{lem:prime-euler-moments}
gives this count when $\ell\nmid q$.
When $\ell\mid q$, the polynomial
is linear.
Let $\Omega_\ell$ denote these classes.
For $\ell\nmid2P$, take $\Gamma_\ell=\Omega_\ell$ with probability $2/(\ell+1)$ and $\Gamma_\ell=\varnothing$ otherwise, independently at each prime.
Let $\Gamma_\ell=\varnothing$ for $\ell\mid2P$.
With $h$ as the sieve variable, the expected density is $g(\ell)=\frac{\ell-1}{\ell(\ell+1)}$ for $\ell\nmid2P$.
Let $g(\ell)=0$ for $\ell\mid2P$.
As in the proof of Lemma \ref{lem:prime-euler-moments}, the integers $1\le h<H$ and the random choices define a nonnegative sequence,
indexed by the product of the primes $\ell<z$ whose classes are excluded.
The Chinese remainder theorem gives $|\mathcal A_d|=Hg(d)+O(1)$ for $d\mid\mathcal P(z)$, since the expected number of excluded classes modulo $d$ is at most one.
Here $0\le g(\ell)\le1/\ell$, so
\eqref{eq:beta-dimension} holds with $\kappa=1$ and an absolute $L$.
Moreover,
$$
 V(z)=\prod_{\ell<z}(1-g(\ell))
 \ll\frac{P}{\varphi(P)}\frac1{\log(2z)}.
$$
This follows from $g(\ell)=1/\ell+O(\ell^{-2})$ off $2P$ and Mertens' product estimate.
Apply Lemma \ref{lem:upper-bound-sieve} with $s=2$ and $D=z^2$.
As above, $E_z(m_h)/\mathfrak M(z)$ is at most the probability of avoiding the classes $\Gamma_\ell$ for every $\ell<z$.
Multiplying the resulting bound by
$\prod_{\ell\le z}(1-1/\ell)^{-1}\ll\log(2z)$ gives
\begin{equation}
\label{eq:individual-euler-average}
 \sum_{1\le h<H}E_z(m_h)
 \ll\frac P{\varphi(P)}H+z^2\log(2z).
\end{equation}
For $H\ge16$, take $z=H^{1/4}$ and $V=16(qH)^2$.
Since $\sqrt{m_h}\le P/(2q)$, Lemma~\ref{lem:short-euler-upper} and \eqref{eq:individual-euler-average} prove \eqref{eq:individual-shell-bound}.
For $1<H<16$, the same lemma with $z=2$ proves the bound.

Finally, let $L=\log q$ and $k_0=\lfloor2\log_pq\rfloor+1$.
Then $\mathcal B_q(k)=0$ for $k<k_0$, and $H_{q,k_0+i}=p^iH_{q,k_0}$ with $1<H_{q,k_0}\le p$.
Thus \eqref{eq:individual-shell-bound} gives
$$
\sum_{k\le K_q}\mathcal B_q(k)
 \ll_p L\sum_{i=0}^{K_q}\frac1{i+1}
 \ll_p L\log L.
 $$
Equation~\eqref{eq:return-reduction} consequently yields $D_q\ll_{p,c_0,c_1}\log\log q$.

\emph{The lower bound.} The error in \eqref{eq:spectral-return}, relative
to $(p+1)p^{2k-1}/N_q$, is $O_p((k+1)N_qp^{-k})$.
This tends to zero uniformly for $k>K_q$. Positivity therefore gives
$$
\Delta_q(\alpha_{q,c})\ge
 \frac{p+1}{pN_q}\,
 \frac{\mathcal M_p(\alpha_{q,c})t^{2K_q+2}}{1-t^2}.
$$
By \eqref{eq:critical-relations} and \eqref{eq:critical-t}, the ratio
$\mathcal M_p/(1-t^2)=((1+t/p)(1+t))^{-1}$ and the factor $t^{2K_q+2}$ are bounded below uniformly for $c\in I$.
This proves the first inequality in \eqref{eq:individual-return-bound}.
\end{proof}

\begin{proof}[Proof of Theorem~\ref{thm:main-results}\textup{(ii)}]
The lower bound is part~\textup{(i)}. For the upper bound, let
$L=\log q$, $Y=q^6$, and $z=(\log Y)^8$.
For $k\le K_q$, write $P=p^k$, $H=P/q^2$, and $m_h=h(P-q^2h)$.
The sum defining $\mathcal B_q(k)$ is empty when $H\le1$.
Otherwise, for $1\le h<H$, $0<m_h\le\frac{P^2}{4q^2} \le\frac{q^6}{4}=\frac Y4$, and $\sqrt{m_h}\le\frac P{2q}$, since $P\le q^4$.
Lemma~\ref{lem:grh-short-euler} and
\eqref{eq:individual-euler-average} therefore give
\begin{equation}
\label{eq:grh-single-exponent}
 \mathcal B_q(k)
 =\frac{q^3}{P^2}\sum_{1\le h<H}r_3(m_h)
 \ll\frac1H\sum_{1\le h<H}E_z(m_h)
 \ll_p 1+\frac{z^2\log(2z)}H.
\end{equation}
Here $P/\varphi(P)=p/(p-1)$.
It follows that $\mathcal B_q(k)\ll_p1$ when $H\ge z^4$.

For $1<H<z^4$, Mertens' product estimate gives $E_z(m_h)\le\prod_{\ell\le z}(1-1/\ell)^{-1}
 \ll\log(2z)$.
The second line of \eqref{eq:grh-single-exponent} then gives $\mathcal B_q(k)\ll\log(2z)$.
There are $O_p(\log(2z))$ such exponents, because
$H_{q,k+1}=pH_{q,k}$.
The full range $k\le K_q$ contains $O_p(L)$ exponents, so
\begin{equation*}
 \sum_{k\le K_q}\mathcal B_q(k)
 \ll_p L+(\log(2z))^2
 \ll_p L+(\log L)^2
 \ll_p L.
\end{equation*}
Inserting this into Lemma~\ref{lem:return-reduction} proves \eqref{eq:grh-return}, uniformly for $c\in I$.
For this level, the proof uses GRH only for $|D|\le q^6$.
\end{proof}

We next average the arithmetic sums over their exponents.

\begin{proposition}
\label{prop:exponent-moments}
For every fixed positive integer $j$,
\begin{equation}
\label{eq:exponent-moments}
 \sum_{q\in\Acal(Q)}
 \left(\frac1{\log q}\sum_{k\le K_q}\Bcal_q(k)\right)^j
 \ll_{p,j}\frac Q{\log Q}.
\end{equation}
\end{proposition}

\begin{proof}
Let $K^*=\lfloor4\log_p(2Q)\rfloor$.
Positivity and H\"older's inequality give
$$
 \left(\sum_{k\le K_q}\Bcal_q(k)\right)^j
 \le(K^*)^{j-1}\sum_{k\le K^*}\Bcal_q(k)^j.
$$
Sum over $q\in\Acal(Q)$ and apply
Proposition~\ref{prop:prime-shell-moments} to each $k$.
Since $K^*\ll_p\log Q$ and $\log q\asymp\log Q$, this proves \eqref{eq:exponent-moments}.
\end{proof}

\begin{proof}[Proof of Theorem~\ref{thm:main-results}\textup{(iii)}]
Quadratic reciprocity identifies the admissible primes, apart from finitely many of them, with $(p-1)/2$ reduced residue classes modulo $4p$.
The prime number theorem in arithmetic progressions
\cite{MontgomeryVaughan2006} gives
$|\Acal(Q)|\sim\frac Q{4\log Q}$, and $|\Acal\cap[1,Y]|\sim\frac Y{4\log Y}$.
Apply Lemma~\ref{lem:return-reduction} and Proposition~\ref{prop:exponent-moments}.
These prove \eqref{eq:main-moments} when $a$ is an integer.
For real $a\ge1$, take an integer $j\ge a$ and apply H\"older's inequality to the average over $\Acal(Q)$.
\end{proof}

\begin{proof}[Proof of Corollary~\ref{cor:density-one}]
Theorem~\ref{thm:main-results}\textup{(iii)} and Markov's inequality give \eqref{eq:exceptional-levels}.
Summing the dyadic moment bounds also gives, for every fixed $a\ge1$,
\begin{equation}
\label{eq:cumulative-moments}
 \sum_{\substack{q\in\Acal\\q\le Y}}D_q^a
 \ll_{p,I,a}\frac Y{\log Y}.
\end{equation}
Finitely many initial levels at which the parameter is undefined are omitted from this sum.

Let $\Acal(w)=\{q\in\Acal:D_q\le w(q)\}$, with those initial levels also omitted.
For fixed $Y_0$ and $Y>Y_0$, monotonicity of $w$ and \eqref{eq:cumulative-moments} imply
$$
 \#\{q\in\Acal:Y_0<q\le Y,\ D_q>w(q)\}
 \le w(Y_0)^{-a}\sum_{\substack{q\in\Acal\\q\le Y}}D_q^a
 \ll_{p,I,a}w(Y_0)^{-a}\frac Y{\log Y}.
$$
Divide by $|\Acal\cap[1,Y]|$, first let $Y\to\infty$, and then let $Y_0\to\infty$.
This proves relative density one and \eqref{eq:density-one}.
\end{proof}

\begin{remark}[PageRank near the source]\label{rem:local}
Fix a nonnegative integer $r$. For all sufficiently large $q$, all vertices $x,y$ with $d(x,y)=r$, and all $c\in I$, one has
\begin{equation}
\label{eq:local-comparison}
 0\le\pi_{\alpha_{q,c},x}(y)-\mathcal M_p(\alpha_{q,c})u_{q,c}^r
 \ll_{p,I,r}\Delta_q(\alpha_{q,c}).
\end{equation}
Here $d$ is graph distance and $u_{q,c}$ is defined by \eqref{eq:alpha-u}.
The case $r=0$ is the definition of $\Delta_q$.
Suppose $r\ge1$.
Let $s_0=1$, $s_j=(p+1)p^{j-1}$ for $j\ge1$, and average over each distance sphere about $x$.
As $2r<g_q$ for large $q$, these spheres up to radius $r$ have the tree cardinalities.
Let $z=(p+1)/\alpha$ and let $h_j$ be the spherical average of $\pi_{\alpha,x}-\mathcal M_p(\alpha)u^{d(x,\cdot)}$.
The resolvent equation gives
$zh_0=(p+1)h_1$, and $zh_j=h_{j-1}+ph_{j+1}$ for $1\le j<r$.
At the source, $h_0=\Delta_q(\alpha)$.
The path polynomials in the proof of Lemma~\ref{lem:return-reduction} therefore give $s_rh_r=Q_r(z) \Delta_q(\alpha)$.
Each individual excess is nonnegative by the covering path expansion:
the unique path to a chosen lift at distance $r$ contributes $\mathcal M_p(\alpha)u^r$.
It is consequently at most $s_rh_r$.
For fixed $r$, $Q_r((p+1)/\alpha_{q,c})$ is bounded uniformly for $c\in I$, proving \eqref{eq:local-comparison}.
Thus the upper bounds in Theorem~\ref{thm:main-results}\textup{(i)}
and~\textup{(ii)}, the moment estimate in part~\textup{(iii)}, and
Corollary~\ref{cor:density-one} also hold, with constants depending on $r$,
for the largest excess at distance $r$ from any source.
\end{remark}

\section*{Acknowledgements}
S.H. thanks Xiaoyao Zheng for suggesting the study of PageRank in 2016, Xianchang Meng for introducing the Ihara zeta function in a course on probabilistic number theory in 2022, and Chung Pang Mok for recommending Sarnak's Penn State lecture notes and for his collaboration with S.H. on murmurations.

%%%%%%%%%%%%%%%%%%%%%%%%%%%%%%%%%%%%%%%%%%%%%%%%%%%%%%%%%%%%%%
%%%%%                      References                    %%%%%
%%%%%%%%%%%%%%%%%%%%%%%%%%%%%%%%%%%%%%%%%%%%%%%%%%%%%%%%%%%%%%

\appendix
\section{The graph construction}\label{app:construction}

\subsection{The vertices and edges}
Fix primes $p,q\equiv1\pmod4$ with $\left(\frac{p}{q}\right)=-1$ and $q>2\sqrt p$.
Start with the integer four-tuples
$$
 S_p=\{(a_0,a_1,a_2,a_3)\in\mathbb Z^4:
 a_0^2+a_1^2+a_2^2+a_3^2=p,\quad
 a_0>0\text{ odd},\quad a_1,a_2,a_3\text{ even}\}.
$$
Jacobi's four-square theorem gives $8(p+1)$ ordered representations of $p$, with signs counted.
A square is $0$ or $1$ modulo $4$, so exactly one coordinate is odd.
By symmetry, one quarter of these representations have $a_0$ odd.
Half of those have $a_0>0$, and hence $|S_p|=p+1$.
For example,
$$
 S_5=\{(1,\pm2,0,0),(1,0,\pm2,0),(1,0,0,\pm2)\}.
$$

Write $\mathbb F_q=\mathbb Z/q\mathbb Z$ and choose $\nu\in\F_q$ with $\nu^2=-1$, which is possible because $q\equiv1\pmod4$.
Turn each tuple $a=(a_0,a_1,a_2,a_3)$ into the matrix
\begin{equation}\label{eq:quaternion-matrix}
 \rho_q(a)=
 \begin{pmatrix}
 a_0+\nu a_1&a_2+\nu a_3\\
 -a_2+\nu a_3&a_0-\nu a_1
 \end{pmatrix}\pmod q.
\end{equation}
Expanding the determinant gives
$\det\rho_q(a)=a_0^2+a_1^2+a_2^2+a_3^2=p$ modulo $q$.
Thus each matrix is invertible.

The group $\PGL_2(\mathbb F_q)$ consists of invertible $2$ by $2$ matrices over $\mathbb F_q$, with two matrices identified when one is a nonzero scalar multiple of the other.
Write $[M]$ for the class of $M$.
Multiplication of matrices defines multiplication of these classes.
The vertices of $X_q$ are the elements of $\PGL_2(\F_q)$.
For each vertex $x$ and each $a\in S_p$, join $x$ to $x[\rho_q(a)]$.
In the usual notation for a Cayley graph,
$$
 X_q=\operatorname{Cay}\bigl(\PGL_2(\mathbb F_q),
\{[\rho_q(a)]:a\in S_p\}\bigr).
$$
Replacing $(a_1,a_2,a_3)$ by $(-a_1,-a_2,-a_3)$ gives another tuple
in $S_p$.
The two corresponding matrices multiply to $pI$.
Their classes are therefore inverses, and every edge can be used in both directions.

The generators are distinct.
If $\rho_q(a)=c\rho_q(b)$, comparison of determinants gives $c^2=1$, so $c=1$ or $-1$.
The matrix formula then gives $a=\pm b$ modulo $q$.
Every coordinate has absolute value at most $\sqrt p$.
Since $q>2\sqrt p$, the same equality holds over the integers.
The positive first coordinates force the plus sign.
Also, a scalar matrix in \eqref{eq:quaternion-matrix} has $a_1=a_2=a_3=0$ modulo $q$.
The coordinate bound would make these three integers zero, contradicting $a_0^2=p$.
Thus the graph has no loops or repeated edges, and every vertex has exactly $p+1$ neighbors.

The reduction theorem of Lubotzky, Phillips, and Sarnak says that these generators generate $\PGL_2(\mathbb F_q)$ \cite[Proposition~3.3]{LPS1988}.
Every vertex can therefore be reached from the identity, so the graph is connected.
Left multiplication preserves edges and carries any vertex to any other, proving vertex transitivity.
Whether a determinant is a square is unchanged by scaling the matrix.
Each step changes this square class because $\left(\frac{p}{q}\right)=-1$.
The square and nonsquare determinant classes form the two parts of the bipartition.
There are $q^2-1$ choices for the first column of an invertible matrix and $q^2-q$ choices for its second column.
Dividing by the $q-1$ scalar choices gives
$$
|V(X_q)|=\frac{(q^2-1)(q^2-q)}{q-1}
=q(q^2-1).
$$
The Lubotzky--Phillips--Sarnak theorem also gives the Ramanujan bound: every adjacency eigenvalue other than $\pm(p+1)$ has absolute value at most
$2\sqrt p$ \cite{LPS1988}.

\subsection{Paths and quaternions}
To multiply the labels of a path, write a tuple as a quaternion
$\gamma=a_0+a_1i+a_2j+a_3k$.
The multiplication rules are
$i^2=j^2=k^2=-1$, $ij=k=-ji$, $jk=i=-kj$, and $ki=j=-ik$.
We use the same notation $S_p$ for the corresponding quaternions.
The conjugate and norm are
$$
 \overline\gamma=a_0-a_1i-a_2j-a_3k,
 \qquad
 \Nrd(\gamma)=\gamma\overline\gamma
             =a_0^2+a_1^2+a_2^2+a_3^2.
$$
Norms multiply, and the matrix map \eqref{eq:quaternion-matrix} preserves multiplication.
A path from the identity with labels
$\gamma_1,\ldots,\gamma_n\in S_p$ therefore ends at $[\rho_q(\gamma_1\cdots\gamma_n)]$.
Its product quaternion has norm $p^n$.
Following a generator by its conjugate immediately reverses the edge, since $\gamma\overline\gamma=p$ is scalar.

Quaternionic factorization gives a bijection between paths of length $n$ from the identity without immediate reversals and integral quaternions of norm $p^n$ whose first coordinate is odd, whose other coordinates are even, and whose coordinates are not all divisible by $p$, up to an overall sign \cite[Lemma~2.3]{LPSHecke1987}.
Choose the sign making the first coordinate positive.
Dividing a quaternion of norm $p^n$ by a common factor $p^j$ reduces its norm to $p^{n-2j}$.
After all common factors $p$ have been removed, factorization gives a path of length $n-2j$.
These are the shorter paths summed in \eqref{eq:complete-hecke}.

A path closes when its product matrix is scalar modulo $q$.
For a product quaternion, \eqref{eq:quaternion-matrix} is scalar
exactly when $q\mid a_1,a_2,a_3$.
This gives the condition in Lemma~\ref{lem:quaternion-count}.
Left multiplication gives the same description of paths starting at any other vertex.

\subsection{The tree and the Hecke operator}
The same matrix formula can be used over the $p$-adic numbers.
An element of $\mathbb Z_p$ is an expansion $\sum_{r\ge0}d_rp^r$ with $0\le d_r<p$, specified by its compatible residues modulo $p,p^2,\ldots$.
The field $\mathbb Q_p$ also allows finitely many negative powers of $p$.
Since $p\equiv1\pmod4$, a square root of $-1$ modulo $p$ lifts successively modulo every power of $p$, giving a square root in $\mathbb Z_p$.
Use this root in place of $\nu$ in the matrix formula.
The resulting map identifies the quaternion algebra over $\mathbb Q_p$ with $M_2(\mathbb Q_p)$ and the integral coordinates with $M_2(\mathbb Z_p)$.

A lattice in $\mathbb Q_p^2$ is the set of $\mathbb Z_p$-linear combinations of a basis.
Identify two lattices when one is a nonzero scalar multiple of the other.
Join two classes when representatives satisfy $pL\subset L'\subset L$ and $L/L'$ has $p$ elements.
For a fixed $L$, the choices of $L'$ correspond to the lines in $L/pL\simeq\F_p^2$.
There are $(p^2-1)/(p-1)=p+1$ such lines.
Each class has a unique representative contained in $\mathbb Z_p^2$ with at least one vector having a coordinate not divisible by $p$.
Elementary row and column operations give a representative $L=\mathbb Z_pe_1+\mathbb Z_pp^re_2$, where $r\ge0$ and $e_1,e_2$ form a basis of $\mathbb Z_p^2$.
For $r>0$, exactly one neighbor has exponent $r-1$, and the other $p$ neighbors have exponent $r+1$.
Repeatedly moving to the first neighbor gives a unique path to the class with $r=0$.
The graph is therefore a tree, called the Bruhat--Tits tree.
The class of $\mathbb Z_p^2$ is its base vertex.
Its stabilizer is $K_p=\PGL_2(\mathbb Z_p)$, so the vertex set is $\PGL_2(\mathbb Q_p)/K_p$.

Let $\Gamma_p$ be the group generated by the projective classes of the quaternions in $S_p$, where scalar multiples are identified.
The $p+1$ generators carry the base vertex to its $p+1$ neighbors.
Quaternionic factorization shows that exactly one element of $\Gamma_p$ carries the base vertex to each tree vertex
\cite[Lemma~2.3]{LPSHecke1987}.
To reduce an element modulo $q$, choose an integral quaternion representative whose norm is a power of $p$.
Two such representatives differ by a scalar $\pm p^j$.
Since $q\ne p$, their reductions give the same projective matrix class.
This defines a surjection $\Gamma_p\to\PGL_2(\F_q)$.
Let $\Gamma_p(q)$ be its kernel, namely, the elements whose reduction is the identity matrix class.
We obtain
$$
 V(X_q)\simeq\Gamma_p(q)\backslash\Gamma_p
 \simeq\Gamma_p(q)\backslash\PGL_2(\mathbb Q_p)/K_p.
$$
The left quotient identifies tree vertices related by an element of $\Gamma_p(q)$.

On this tree, the spherical Hecke operator at $p$ is the neighbor sum
$$
 (T_pf)(v)=\sum_{w\sim v}f(w).
$$
In group notation it is convolution with the indicator of $K_p\operatorname{diag}(p,1)K_p$, with invariant volume chosen so that $K_p$ has volume one.
This set is the disjoint union of $p+1$ sets $g_iK_p$.
Each has volume one, and the vertices $gg_iK_p$ are the neighbors of $gK_p$.
Hence convolution is exactly the displayed sum.
The sum descends to the finite quotient $X_q$ and gives $A_q$.


\begin{thebibliography}{99}

\bibitem{Bass1992}
Bass, Hyman.
The Ihara--Selberg zeta function of a tree lattice.
\emph{Internat. J. Math.} 3 (1992), no. 6, 717--797.


\bibitem{BerkovichJagy2012}
Berkovich, Alexander; Jagy, William C.
On representation of an integer as the sum of three squares and ternary
quadratic forms with the discriminants $p^2$, $16p^2$.
\emph{J. Number Theory} 132 (2012), no. 1, 258--274.


\bibitem{BiggsBoshier1990}
Biggs, Norman L.; Boshier, Alan G.
Note on the girth of Ramanujan graphs.
\emph{J. Combin. Theory Ser. B} 49 (1990), no. 2, 190--194.


\bibitem{FriedlanderIwaniec}
Friedlander, John B.; Iwaniec, Henryk.
\emph{Opera de cribro}.
American Mathematical Society Colloquium Publications, vol. 57.
American Mathematical Society, Providence, RI, 2010.


\bibitem{Gleich2015}
Gleich, David F.
PageRank beyond the Web.
\emph{SIAM Rev.} 57 (2015), no. 3, 321--363.


\bibitem{GranvilleSoundararajan2003}
Granville, Andrew; Soundararajan, Kannan.
The distribution of values of $L(1,\chi_d)$.
\emph{Geom. Funct. Anal.} 13 (2003), no. 5, 992--1028.


\bibitem{HuaMok2026}
Hua, Shenghao; Mok, Chung Pang.
Universal murmuration and Hecke augmentation.
(Preprint).
arXiv: 2609.00777.


\bibitem{KunduMuller2026}
Kundu, Debanjana; M\"uller, Katharina.
Murmurations of modular forms and $p$-power coefficients.
\emph{Math. Proc. Cambridge Philos. Soc.} 180 (2026), no. 1, 65--92.


\bibitem{LPS1988}
Lubotzky, Alexander; Phillips, Ralph Saul; Sarnak, Peter.
Ramanujan graphs.
\emph{Combinatorica} 8 (1988), no. 3, 261--277.


\bibitem{LPSHecke1987}
Lubotzky, Alexander; Phillips, Ralph Saul; Sarnak, Peter.
Hecke operators and distributing points on $S^2$. II.
\emph{Comm. Pure Appl. Math.} 40 (1987), no. 4, 401--420.


\bibitem{McKay1981}
McKay, Brendan D.
The expected eigenvalue distribution of a large regular graph.
\emph{Linear Algebra Appl.} 40 (1981), 203--216.


\bibitem{McKay1983}
McKay, Brendan D.
Spanning trees in regular graphs.
\emph{European J. Combin.} 4 (1983), no. 2, 149--160.


\bibitem{MontgomeryVaughan2006}
Montgomery, Hugh L.; Vaughan, Robert C.
\emph{Multiplicative Number Theory. I. Classical Theory}.
Cambridge Studies in Advanced Mathematics, vol. 97.
Cambridge University Press, Cambridge, 2006.


\bibitem{NestoridiSarnak2023}
Nestoridi, Evita; Sarnak, Peter.
Bounded cutoff window for the non-backtracking random walk on Ramanujan graphs.
\emph{Combinatorica} 43 (2023), no. 2, 367--384.


\bibitem{PageEtAl1999}
Page, Lawrence; Brin, Sergey; Motwani, Rajeev; Winograd, Terry.
The PageRank citation ranking: bringing order to the Web.
Stanford InfoLab Technical Report 1999-66, 1999.


\bibitem{Sarnak2025Lectures}
Sarnak, Peter.
Penn State lectures: Non-tame distribution of zeros of zeta functions;
spectra of surfaces and graphs.
Lecture notes, October 2025, with an addendum on an Archimedean McKay law.
\url{https://publications.ias.edu/sarnak/paper/2741}.


\bibitem{SarnakZubrilina2024}
Sarnak, Peter; Zubrilina, Nina.
Convergence to the Plancherel measure of Hecke eigenvalues.
\emph{Acta Arith.} 214 (2024), 191--213.


\bibitem{Serre1997}
Serre, Jean-Pierre.
R\'epartition asymptotique des valeurs propres de l'op\'erateur de Hecke $T_p$.
\emph{J. Amer. Math. Soc.} 10 (1997), no. 1, 75--102.


\bibitem{Zubrilina2025}
Zubrilina, Nina.
Murmurations.
\emph{Invent. Math.} 241 (2025), no. 3, 627--680.

\end{thebibliography}
\end{document}